\documentclass{scrartcl}

\usepackage[naustrian, english]{babel}
\usepackage{hyperref}
\usepackage[T1]{fontenc}
\usepackage{lmodern}
\usepackage[utf8]{inputenc}
\usepackage{amsmath}
\usepackage{amssymb}
\usepackage{amsthm}
\usepackage[babel]{csquotes}
\usepackage{tikz}
\usepackage{enumitem}
\usepackage{bm}
\usepackage{upgreek}
\usepackage{subfigure}

\newcommand{\p}[1]{\ensuremath{\mathord{\left(#1\right)}}}

\newcommand{\set}[1]{\ensuremath{\left\lbrace #1 \right\rbrace}}
\newcommand{\setcond}[2]{\ensuremath{\left\lbrace #1 \,\middle\vert \, #2 \right\rbrace}}
\newcommand{\inpr}[2]{\ensuremath{\left\langle #1, #2 \right\rangle}}
\newcommand{\RR}{\ensuremath{\mathbb{R}}}

\newcommand{\defeq}{\ensuremath{\mathrel{\mathop:}=}}

\newcommand{\RRc}{\ensuremath{\overline{\mathbb{R}}}}
\newcommand{\toset}{\ensuremath{\rightrightarrows}}
\newcommand{\Hh}{\mathcal{H}}
\newcommand{\Gg}{\mathcal{G}}
\newcommand{\NN}{\mathbb{N}}
\newcommand{\BB}{\mathbb{B}}

\newcommand{\norm}[1]{\left\Vert #1 \right\Vert}
\newcommand{\weakto}{\rightharpoonup}
\newcommand{\bop}{\mathop{\Box}}

\newcommand{\Id}{\ensuremath{\mathrm{Id}}}

\DeclareMathOperator{\dom}{dom}
\DeclareMathOperator{\zer}{zer}

\DeclareMathOperator*{\argmin}{arg\,min}
\DeclareMathOperator{\Graph}{Graph}
\DeclareMathOperator{\Dom}{Dom}
\DeclareMathOperator{\Ran}{Ran}
\DeclareMathOperator{\Proj}{Proj}

\theoremstyle{plain}
\ifcsdef{lemma}{\relax}{\newtheorem{lemma}{Lemma}}
\ifcsdef{theorem}{\relax}{\newtheorem{theorem}{Theorem}}
\ifcsdef{corollary}{\relax}{}

\theoremstyle{definition}
\ifcsdef{definition}{\relax}{}
\ifcsdef{example}{\relax}{\newtheorem{example}{Example}}
\ifcsdef{remark}{\relax}{\newtheorem{remark}{Remark}}
\ifcsdef{problem}{\relax}{\newtheorem{problem}{Problem}}
\ifcsdef{algorithm}{\relax}{\newtheorem{algorithm}{Algorithm}}

\title{Backward Penalty Schemes for Monotone Inclusion Problems}
\author{Sebastian Banert\thanks{University of Vienna, Oskar-Morgenstern-Platz 1, 1090 Vienna, Austria, sebastian.banert@univie.ac.at} \and 
Radu Ioan Bo\c{t}\thanks{University of Vienna, Faculty of Mathematics, Oskar-Morgenstern-Platz 1, 1090 Vienna, Austria, radu.bot@univie.ac.at. 
Research partially supported by DFG (German Research Foundation), project BO 2516/4-1}}

\begin{document}
\maketitle

\noindent \textbf{Abstract.} In this paper we are concerned with solving monotone inclusion problems expressed by the sum of a 
set-valued maximally monotone operator with a single-valued maximally monotone one and the normal cone to the nonempty set of zeros of another set-valued maximally monotone operator.
Depending on the nature of the single-valued operator, we will propose two iterative penalty schemes, both addressing the set-valued operators via backward steps. The
single-valued operator will be evaluated via a single forward step if it is cocoercive, and via two forward steps if it is monotone and Lipschitz continuous. The latter situation represents 
the starting point for dealing with complexly structured monotone inclusion problems from algorithmic point of view.\vspace{1ex}

\noindent \textbf{Key Words.} backward penalty algorithm, monotone inclusion, maximally monotone operator, Fitzpatrick function, subdifferential\vspace{1ex}

\noindent \textbf{AMS subject classification.} 47H05, 65K05, 90C25

\section{Introduction and preliminaries}
\subsection{Motivation}
In this article we address the solving of variational inequalities expressed as monotone inclusion problems of the form 
\begin{equation*}
    0 \in Ax + Dx + N_C\p{x},
 \end{equation*}
where $\Hh$ is a real Hilbert space, $A, B: \Hh \toset \Hh$ are (set-valued) maximally monotone operators, $D: \Hh \to \Hh$  is a (single-valued) maximally monotone operator, 
 $C := \zer B \neq \emptyset$ and $N_C$ denotes the normal cone to the set $C$. These investigations complement the ones
made in \cite{BotCsetnek:2013a} for solving monotone inclusion problems of the same form, however, whenever $B$ is a single-valued maximally monotone operator.
 
For the beginning we assume that $D$ is cocoercive and propose an iterative scheme of penalty type that evaluates $D$ via a single forward step and 
the operators $A$ and $B$ via their resolvents. Under some hypotheses expressed in terms of the Fitzpatrick function associated to the operator $B$, we prove 
weak ergodic convergence for the sequence of generated iterates, but also strong convergence, provided that $A$ is strongly monotone. If $Dx = 0$ for all $x \in \Hh$ and $B = \partial \Psi$, 
where $\Psi : \Hh \rightarrow \RRc$ is a proper, convex and lower semicontinuous function with $\min \Psi = 0$, then the iterative scheme
reduces to the algorithm proposed and investigated in \cite{AttouchCzarneckiPeypouquet:2011a} for solving the monotone inclusion problem
\begin{equation*}
    0 \in Ax + N_{\argmin \Psi}\p{x}.
 \end{equation*}

Further, by assuming that $D$ is (only) monotone and Lipschitz continuous, we provide a second iterative scheme that also addresses $A$ and $B$ via their resolvents, while $D$ is evaluated via two
forward steps. For this scheme a convergence analysis is undertaken, as well, by proving for the generated sequences of iterates weak ergodic convergence and, whenever $A$ is strongly monotone, weak convergence.
The iterative scheme and the convergence statements provided in this context constitute the starting point for solving complexly structured variational inequalities, involving mixtures of sums of maximally monotone
operators and linear compositions of parallel sums of maximally monotone operators.

We close the paper by discussing the fulfillment of the assumption expressed via the Fitzpatrick function associated to the operator $B$ for some particular instances of the latter.

\subsection{Notation and preliminary results}
For the reader's convenience we present first some notations which are used throughout the paper (see \cite{BauschkeCombettes:2011, BorweinVanderwerff:2010, Bot:2010, EkelandTemam:1999, Simons:2008, Zalinescu:2002}). By $\NN = \set{1, 2, \ldots}$ we denote the set of \emph{positive integer numbers}. Let $\Hh$ be a real Hilbert space with \emph{inner product} $\inpr{\cdot}{\cdot}$ and associated \emph{norm} $\norm{\cdot} = \sqrt{\inpr{\cdot}{\cdot}}$. The symbols $\weakto$ and $\to$ denote weak and strong convergence, respectively. When $\Gg$ is another Hilbert space and $K: \Hh \to \Gg$ is a continuous linear operator, then the \emph{norm} of $K$ is defined as $\norm{K} = \sup\setcond{\norm{Kx}}{x\in\Hh, \norm{x} \leq 1}$, while $K^*: \Gg \to \Hh$, defined by $\inpr{K^* y}{x} = \inpr{y}{Kx}$ for all $\p{x, y}\in \Hh \times \Gg$, denotes the \emph{adjoint operator} of $K$.

For a function $f: \Hh \to \RRc$ we denote by $\dom f = \setcond{x\in \Hh}{f\p{x} < +\infty}$ its \emph{effective domain} and say that $f$ is \emph{proper} if 
$\dom f \neq \emptyset$ and $f\p{x} \neq -\infty$ for all $x\in\Hh$. Let $f^*: \Hh \to \RRc$, $f^*\p{u} = \sup\setcond{\inpr{u}{x} - f\p{x}}{x\in\Hh}$ for all $u\in\Hh$, be the \emph{conjugate function} of $f$. 
The \emph{subdifferential} of $f$ at $x\in\Hh$, with $f\p{x}\in\RR$, is the set $\partial f\p{x} \defeq \setcond{v\in\Hh}{f\p{y} \geq f\p{x} + \inpr{v}{y-x} \text{ for all } y\in\Hh}$. 
We take by convention $\partial f\p{x} \defeq \emptyset$ if $f\p{x} \in \set{\pm\infty}$. 
We also denote by $\argmin f$ the set of global minima of the function $f$ and set $\min f \defeq \inf \{f(x) | x \in \argmin f\}$. The \emph{infimal convolution} of two functions $f, g: \Hh \to \RRc$ is defined as
\[
  \p{f\bop g}\p{x} \defeq \inf \setcond{f\p{y} + g\p{x - y}}{y \in \Hh},
\]
and we have $\p{f \bop g}^* = f^* + g^*$.

Let $S \subseteq \Hh$ be a nonempty set. The \emph{indicator function} of $S$, $\delta_S: \Hh \to \RRc$, is the function which takes the value $0$ on $S$ and $+\infty$ elsewhere. 
The subdifferential of the indicator function is the \emph{normal cone} of $S$, that is, $N_S\p{x} = \setcond{u\in\Hh}{\inpr{u}{y-x} \leq 0 \text{ for all } y\in S}$ if $x\in S$ and $N_S\p{x} = 
\emptyset$ for $x\notin S$. Notice that for $x\in S$, $u\in N_S\p{x}$ if and only if $\sigma_S\p{u} = \inpr{u}{x}$, where $\sigma_S$ is the \emph{support function} of $S$, defined by $\sigma_S\p{u} = 
\sup \setcond{\inpr{u}{y}}{y\in S}$.

For an arbitrary set-valued operator $M: \Hh \toset \Hh$ we denote by $\Graph M = \{(x, u) \in \Hh \times \Hh | u\in Mx\}$ its \emph{graph}, 
by $\Dom M = \setcond{x\in\Hh}{Mx\neq\emptyset}$ its \emph{domain}, by $\Ran M = \bigcup\setcond{Mx}{x\in\Hh}$ its \emph{range} and by 
$M^{-1}: \Hh \toset \Hh$ its \emph{inverse operator}, defined by $\p{u, x}\in \Graph M^{-1}$ if and only if $\p{x, u}\in \Graph M$. 

We also use the notation $\zer M = \setcond{x\in\Hh}{0\in Mx}$ for the \emph{set of zeros} of the operator $M$. We say that $M$ is \emph{monotone} if 
$\inpr{x - y}{u - v} \geq 0$ for all $\p{x, u}, \p{y, v} \in \Graph M$. A monotone operator $M$ is said to be \emph{maximally monotone} if there exists no proper monotone extension of the graph of $M$ on 
$\Hh \times \Hh$. Let us mention that in case $M$ is maximally monotone, on has the following characterization for the set of its zeros.
\begin{equation}\label{eq:maxmon_zeros}
  z\in \zer M \qquad \text{if and only if} \qquad \inpr{w}{u - z} \geq 0 \text{ for all } \p{u, w}\in\Graph M.
\end{equation}
The operator $M$ is said to be \emph{strongly monotone with parameter $\gamma > 0$} or \emph{$\gamma$-strongly monotone}, if $\inpr{x - y}{u - v} \geq \gamma \norm{x - y}^2$ for all $\p{x, u}, \p{y, v} \in \Graph{M}$. 
Notice that if $M$ is maximally monotone and strongly monotone (with a given parameter), then $\zer M$ is a singleton, thus nonempty (see \cite[Corollary 23.37]{BauschkeCombettes:2011}).

The \emph{resolvent} of $M$, $J_M : \Hh \toset \Hh$, is defined by $J_M = \p{\Id + M}^{-1}$, where $\Id: \Hh \to \Hh$, $\Id\p{x} = x$ for all $x\in\Hh$, denotes the \emph{identity operator} on $\Hh$. If 
$M$ is maximally monotone, then $J_M: \Hh \to \Hh$ is single-valued and maximally monotone (cf. \cite[Proposition 23.7 and Corollary 23.10]{BauschkeCombettes:2011}). For an arbitrary $\gamma > 0$ we have (see \cite[Proposition 23.18]{BauschkeCombettes:2011})
\begin{equation}\label{eq:maxmon_resolvent_decomposition}
  J_{\gamma M} + \gamma J_{\gamma^{-1} M^{-1}} \circ \gamma^{-1} \Id = \Id.
\end{equation}

For the convergence statements that we provide in this paper we will assume that some hypotheses, one of them expressed in terms of the Fitzpatrick function associated to a certain maximally monotone operator, are fulfilled. 
In the following we will recall some properties of this function, which brought new and deep insights into the field of maximally monotone operators in the last decade
(see \cite{BauschkeCombettes:2011, BauschkeMclarenSendov:2006, Borwein:2006, BorweinVanderwerff:2010, Bot:2010, BotCsetnek:2008, BurachikSvaiter:2002, Fitzpatrick:1988, Simons:2008} and the references therein). 
The \emph{Fitzpatrick function} associated to a monotone operator $M$, defined as
\[
  \varphi_M: \Hh \times \Hh \to \Hh \to \RRc, \qquad \varphi_M\p{x, u} = \sup\setcond{\inpr{x}{v} + \inpr{y}{u} - \inpr{y}{v}}{\p{y, v}\in \Graph M},
\]
is a convex and lower semicontinuous function. In case $M$ is maximally monotone, $\varphi_M$ is proper and it fulfills
\[
  \varphi_M\p{x, u} \geq \inpr{x}{u} \qquad \text{for all } \p{x, u}\in \Hh \times \Hh,
\]
with equality if and only if $\p{x, u}\in\Graph M$. Notice that if $f: \Hh \to \RRc$ is a proper, convex and lower semicontinuous function, then $\partial f$ is a maximally monotone operator (cf. \cite{Rockafellar:1970a}) and it holds $\p{\partial f}^{-1} = \partial f^*$. Furthermore, the following inequality is true (see \cite{BauschkeMclarenSendov:2006})
\begin{equation}\label{eq:subdifferential_Fitzpatrick}
  \varphi_{\partial f} \p{x, u} \leq f\p{x} + f^*\p{u} \qquad \text{for all } \p{x, u}\in \Hh \times \Hh.
\end{equation}
We refer the reader to \cite{BauschkeMclarenSendov:2006} for formulas of the corresponding Fitzpatrick functions computed for particular classes of monotone operators.

Let $\gamma > 0$ be arbitrary. A single-valued operator $M: \Hh \to \Hh$ is said to be \emph{$\gamma$-cocoercive} if $\inpr{x - y}{Mx - My} \geq \gamma \norm{Mx - My}^2$ for all $\p{x, y}\in \Hh \times \Hh$, and \emph{$\gamma$-Lipschitz continuous} if $\norm{Mx - My} \leq \gamma \norm{x - y}$ for all $\p{x, y}\in \Hh\times \Hh$. A single-valued linear operator $M: \Hh \to \Hh$ is said to be \emph{skew} if $\inpr{x}{Mx} = 0$ for all $x\in\Hh$.

We close the section by presenting some convergence results which will be used when carrying out a convergence analysis for the iterative schemes provided in the paper. 
Let $\p{x_n}_{n \in \NN}$ be a sequence in $\Hh$ and $\p{\lambda_k}_{k \in \NN}$ be a sequence of positive numbers such that $\sum_{k \in \NN} \lambda_k = +\infty$. 
Let $\p{z_n}_{n \in \NN}$ be the sequence of weighted averages defined as (see \cite{AttouchCzarneckiPeypouquet:2011b})
\begin{equation}\label{eq:zndef}
  z_n = \frac{1}{\tau_n} \sum_{k=1}^n \lambda_k x_k, \qquad \text{where } \tau_n = \sum_{k=1}^n \lambda_k \qquad \text{for all } n \in \NN.
\end{equation}

\begin{lemma}[Opial--Passty]\label{lem:OpialPassty}
Let $F$ be a nonempty subset of $\Hh$ and assume that the limit $\lim_{n\to +\infty} \norm{x_n - x}$ exists for every $x\in F$. 

(i) If every weak cluster point of $\p{x_n}_{n \geq 0}$  lies in $F$, then $\p{x_n}_{n \geq 0}$ converges weakly to an element in $F$ as $n\to +\infty$.

(ii) If every weak cluster point of  $\p{z_n}_{n \geq 0}$ lies in $F$, then $\p{z_n}_{n \geq 0}$ converges weakly to an element in $F$ as $n\to +\infty$.
\end{lemma}

The following result is taken from \cite{AttouchCzarneckiPeypouquet:2011b}.
\begin{lemma}\label{lem:summability}
  Let $\p{a_n}_{n \geq 0}$, $\p{b_n}_{n \geq 0}$ and $\p{\varepsilon_n}_{n \geq 0}$ be real sequences. Assume that $\p{a_n}_{n \geq 0}$ is bounded from below, $\p{b_n}_{n \geq 0}$ is nonnegative, 
  $\p{\varepsilon_n}_{n \geq 0} \in \ell^1$ and $a_{n+1} - a_n + b_n \leq \varepsilon_n$ for any $n \geq 0$. Then $\p{a_n}_{n \geq 0}$ is convergent and $\p{b_n}_{n \geq 0} \in \ell^1$.
\end{lemma}

\section{A backward penalty scheme with one forward step}
The problem we deal with in this section has the following formulation.
\begin{problem}\label{prob:FBB}
  Let $\Hh$ be a real Hilbert space, $A, B: \Hh \toset \Hh$ be maximally monotone operators, $D: \Hh \to \Hh$ be an $\eta$-cocoercive operator with $\eta > 0$ and suppose that 
  $C := \zer B \neq \emptyset$. The monotone inclusion problem to solve is
 \begin{equation}\label{tosolve}
    0 \in Ax + Dx + N_C\p{x}.
 \end{equation}
\end{problem}

We propose for solving Problem \ref{prob:FBB} the following iteration scheme which has the particularity that it evaluates an appropriate penalization of the operator $B$ via a backward step.

\begin{algorithm}\label{alg:FBB}
  Choose $x_0 \in \Hh$ and set for any $n \geq 1$:
  \begin{align*}
    y_{n-1} &= x_{n-1} - \lambda_n Dx_{n-1}, \\
    w_n &= J_{\lambda_n A} y_{n-1}, \\
    x_{n} &= J_{\lambda_n \beta_n B} w_n,
  \end{align*}
  where $\p{\lambda_n}_{n \in \NN}$ and $\p{\beta_n}_{n \in \NN}$ are sequences of positive real numbers. 
\end{algorithm}

\begin{remark}\label{remark1}
(a) If $Dx = 0$ for all $x \in \Hh$ and $B = \partial \Psi$, where $\Psi : \Hh \rightarrow \RRc$ is a proper, convex and lower semicontinuous function with $\min \Psi = 0$, then the iterative scheme
in Algorithm \ref{alg:FBB} reduces to the algorithm proposed and investigated in \cite{AttouchCzarneckiPeypouquet:2011a} for solving the monotone inclusion problem
\begin{equation}\label{tosolvepart}
    0 \in Ax + N_{\argmin \Psi}\p{x}.
 \end{equation}

(b) Another penalty scheme for solving the monotone inclusion problem \eqref{tosolve}, in case $B$ is a cocoercive operator, which evaluates both $B$ and $D$ via forward steps and $A$ via a backward step 
has been introduced and investigated from the point of view of its convergence properties in \cite{BotCsetnek:2013a}. The mentioned algorithm is an extension of a numerical method proposed in 
\cite{AttouchCzarneckiPeypouquet:2011b} in the context of solving \eqref{tosolvepart} when $\Psi$ is, additionally, differentiable with Lipschitz continuous gradient.
\end{remark}

The following lemma will be crucial for proving the convergence of Algorithm \ref{alg:FBB}.

\begin{lemma}\label{lem:FBB:inequality}
  For $u\in C\cap \dom A$ take $w \in \p{A + D + N_C}\p{u}$ such that $w = v + Du + p$ for some $v\in Au$ and $p\in N_C\p{u}$. For each $n \in \NN$, the following inequality holds:
 \begin{align*}
 & \norm{x_{n} - u}^2 - \norm{x_{n-1} - u}^2 + \lambda_n\p{2\eta - \lambda_n} \norm{Dx_{n-1} - Du}^2 + \\
 &  \frac{1}{2}\norm{x_{n} - w_n}^2 + \frac{1}{2}\norm{x_{n} - w_n - \lambda_n \p{Du + v}}^2 + \norm{x_{n-1} - w_n - \lambda_n\p{Dx_{n-1} - Du}}^2  \leq\\
 & 2\lambda_n \beta_n \p{\sup_{\tilde u\in C} \varphi_B\p{\tilde u, \frac{p}{\beta_n}} - \sigma_C\p{\frac{p}{\beta_n}}} + 2\lambda_n \inpr{w}{u - x_{n}} + 2\lambda_n^2 \norm{Du + v}^2. 
\end{align*}
\end{lemma}
\begin{proof}
Let be $n \geq 1$ fixed.

We have $\lambda_n v \in \lambda_n Au$ and $y_{n-1} - w_n \in \lambda_n A w_n$, so, by monotonicity of $A$,
\begin{equation}\label{Amon}
  \inpr{y_{n-1} - w_n - \lambda_n v}{w_n - u} \geq 0,
\end{equation}
which is equivalent to
\begin{equation}\label{eq:FBB:monA}
  2\lambda_n \inpr{v}{w_n - u} \leq 2\inpr{y_{n-1} - w_n}{w_n - u} = \norm{y_{n-1} - u}^2 - \norm{y_{n-1} - w_n}^2 - \norm{u - w_n}^2.
\end{equation}
Furthermore, we have $w_n - x_{n} \in \lambda_n \beta_n Bx_{n}$, so, by definition of the Fitzpatrick function,
\begin{align}
  &\mathrel{\phantom{=}} 2\lambda_n \beta_n \p{\sup_{\tilde u\in C} \varphi_B\p{\tilde u, \frac{p}{\beta_n}} - \sigma_C\p{\frac{p}{\beta_n}}} \geq 2\lambda_n \beta_n \p{\varphi_B\p{u, \frac{p}{\beta_n}} - \sigma_C\p{\frac{p}{\beta_n}}} \nonumber \\
  &\geq 2\inpr{u}{w_n - x_{n}} + 2\lambda_n \inpr{p}{x_{n}} - 2\inpr{x_{n}}{w_n - x_{n}} - 2\lambda_n\inpr{p}{u} \nonumber \\
  &= 2\inpr{u - x_{n}}{w_n - x_{n}} + 2\lambda_n\inpr{p}{x_{n} - u} \nonumber \\
  &= \norm{u - x_{n}}^2 + \norm{x_{n} - w_n}^2 - \norm{u - w_n}^2 + 2\lambda_n \inpr{p}{x_{n} - u}. \label{eq:FBB:FitzB}
\end{align}
Adding \eqref{eq:FBB:monA} and \eqref{eq:FBB:FitzB}, we obtain
{\allowdisplaybreaks \begin{align*}
  &\mathrel{\phantom{\leq}}\norm{x_{n} - u}^2 - \norm{x_{n-1} - u}^2 - 2\lambda_n \beta_n \p{\sup_{\tilde u\in C} \varphi_B\p{\tilde u, \frac{p}{\beta_n}} - \sigma_C\p{\frac{p}{\beta_n}}} - 2\lambda_n \inpr{w}{u - x_{n}} \nonumber \\
  &\leq \norm{u - y_{n-1}}^2 - \norm{x_{n-1} - u}^2 - \norm{y_{n-1} - w_n}^2 - \norm{x_{n} - w_n}^2 + 2\lambda_n \inpr{v}{x_{n} - w_n} \nonumber\\ 
  &\mathrel{\phantom{=}}\mathop{+} 2\lambda_n \inpr{Du}{x_{n} - u} \nonumber \\
  &= \norm{u - x_{n-1} + \lambda_n Dx_{n-1}}^2 - \norm{u - x_{n-1}}^2 - \norm{x_{n-1} - w_n - \lambda_n Dx_{n-1}}^2 - \norm{x_{n} - w_n}^2 \nonumber \\
  &\mathrel{\phantom{=}}\mathop{+} 2\lambda_n \inpr{v}{x_{n} - w_n} + 2\lambda_n \inpr{Du}{x_{n} - u} \nonumber \\
  &= 2\lambda_n\inpr{Dx_{n-1}}{u-x_{n-1}} - \norm{x_{n-1} - w_n}^2 + 2\lambda_n\inpr{Dx_{n-1}}{x_{n-1} - w_n} - \norm{x_{n} - w_n}^2 \nonumber \\
  &\mathrel{\phantom{=}}\mathop{+} 2\lambda_n \inpr{v}{x_{n} - w_n} + 2\lambda_n \inpr{Du}{x_{n} - u} \nonumber \\
  &= 2\lambda_n\inpr{Dx_{n-1}}{u - w_n} - \norm{x_{n-1} - w_n}^2 - \norm{x_{n} - w_n}^2 \nonumber \\
  &\mathrel{\phantom{=}}\mathop{+} 2\lambda_n \inpr{v}{x_{n} - w_n} + 2\lambda_n \inpr{Du}{x_{n} - u} \nonumber \\
  &= 2\lambda_n\inpr{Dx_{n-1} - Du}{u - x_{n-1}} + 2\lambda_n \inpr{Dx_{n-1}}{x_{n-1} - w_n} - \norm{x_{n-1} - w_n}^2 \nonumber \\
  &\mathrel{\phantom{=}}\mathop{-}  \norm{x_{n} - w_n}^2  + 2\lambda_n \inpr{v}{x_{n} - w_n} + 2\lambda_n \inpr{Du}{x_{n} - x_{n-1}} \nonumber \\
  &\leq -2\eta\lambda_n\norm{Dx_{n-1} - Du}^2 + 2\lambda_n \inpr{Dx_{n-1} - Du}{x_{n-1} - w_n} - \norm{x_{n-1} - w_n}^2  \nonumber \\
  &\mathrel{\phantom{=}}\mathop{-} \norm{x_{n} - w_n}^2 + 2\lambda_n \inpr{Du + v}{x_{n} - w_n} \nonumber \\
  &= - \norm{x_{n-1} - w_n - \lambda_n\p{Dx_{n-1} - Du}}^2 - \lambda_n\p{2\eta - \lambda_n} \norm{Dx_{n-1} - Du}^2 \nonumber \\
  &\mathrel{\phantom{=}}\mathop{-} \frac{1}{2}\norm{x_{n} - w_n - \lambda_n \p{Du + v}}^2  - \frac{1}{2}\norm{x_{n} - w_n}^2 + 2\lambda_n^2 \norm{Du + v}^2.
\end{align*}}
From here the conclusion is straightforward.
\end{proof}

For the convergence statement of Algorithm \ref{alg:FBB}, the following hypotheses are needed:
\begin{equation}\label{eq:Hfitz}
  \left.\begin{minipage}{12.7cm}
    \begin{enumerate}
      \item [(i)] $A + N_C$ is maximally monotone and $\zer\p{A + D + N_C} \neq \emptyset$;
      \item [(ii)] For every $p\in\Ran N_C$, $\displaystyle\sum_{n \in \NN} \lambda_n \beta_n \p{\sup_{\tilde u\in C} \varphi_B\p{\tilde u, \frac{p}{\beta_n}} - \sigma_C\p{\frac{p}{\beta_n}}} < +\infty$;
      \item [(iii)] $\p{\lambda_n}_{n \in \NN} \in \ell^2 \setminus \ell^1$.
    \end{enumerate}
  \end{minipage}
  \right\rbrace \tag{H\textsubscript{fitz}}
\end{equation}

\begin{remark}\label{remark2}
Some comments with respect to the hypotheses \eqref{eq:Hfitz} are in order.

(a) The hypotheses \eqref{eq:Hfitz} have already been used in \cite{BotCsetnek:2013a} when showing the convergence of the iterative scheme proposed for solving \eqref{tosolve} when $B$ is a cocoercive operator. Still there it was
pointed out that, since $D$ is cocoercive and $\dom D = \Hh$, $A + D + N_C$ is maximally monotone, while, in the light of the properties of the Fitzpatrick function, for every $p\in\Ran N_C$ and any $n \in \NN$ one has
$$\sup_{\tilde u\in C} \varphi_B \left(\tilde u, \frac{p}{\beta_n} \right) - \sigma_C \left (\frac{p}{\beta_n} \right) \geq 0.$$

(b) The convergence of the penalty iterative scheme proposed in \cite{AttouchCzarneckiPeypouquet:2011a} for solving the monotone inclusion problem \eqref{tosolvepart}, 
where  $\Psi : \Hh \rightarrow \RRc$ is a proper, convex and lower semicontinuous function with $\min \Psi = 0$,  have been shown under the following hypotheses:
\begin{equation}\label{eq:Hconj}
  \left.\begin{minipage}{12.7cm}
    \begin{enumerate}
      \item [(i)] $A + N_C$ is maximally monotone and $\zer\p{A + D + N_C} \neq \emptyset$;
      \item [(ii)] For every $p\in\Ran N_C$, $\displaystyle\sum_{n\in \NN} \lambda_n \beta_n \p{\Psi^*\p{\frac{p}{\beta_n}} - \sigma_C\p{\frac{p}{\beta_n}}} < +\infty$;
      \item [(iii)] $\p{\lambda_n}_{n \in \NN} \in \ell^2 \setminus \ell^1$.
    \end{enumerate}
  \end{minipage}
  \right\rbrace \tag{H}
\end{equation}
According to \eqref{eq:subdifferential_Fitzpatrick} it holds
\begin{equation*}
\sum_{n\in \NN} \lambda_n \beta_n \p{\sup_{\tilde u\in C} \varphi_{\partial \Psi}\p{\tilde u, \frac{p}{\beta_n}} - \sigma_C\p{\frac{p}{\beta_n}}} \leq \sum_{n \in \NN} \lambda_n \beta_n \p{\Psi^*\p{\frac{p}{\beta_n}} - 
\sigma_C\p{\frac{p}{\beta_n}}},
\end{equation*}
thus condition (ii) in \eqref{eq:Hconj} implies condition (ii) in \eqref{eq:Hfitz} applied to $B = \partial \Psi$. This shows that the hypothesis formulated by means of the Fitzpatrick function
extends the one from \cite{AttouchCzarneckiPeypouquet:2011a} to the more general setting considered in Problem \ref{prob:FBB}. In the last section of this paper we will discuss the fulfillment of
the hypotheses \eqref{eq:Hconj} and \eqref{eq:Hfitz} for different particular instances of the operator $B$.
\end{remark}

\begin{theorem}\label{thm:FBB}
Let $\p{x_n}_{n\geq 0}$ and $\p{w_n}_{n \in \NN}$ be the sequences generated by Algorithm \ref{alg:FBB} and $\p{z_n}_{n \in \NN}$ be the sequence defined in \eqref{eq:zndef}. 
If \eqref{eq:Hfitz} is fulfilled, then $\p{z_n}_{n \in \NN}$ converges weakly to an element in $\zer\p{A+D+N_C}$ as $n\to +\infty$.
\end{theorem}

\begin{proof}
As $\lim_{n\to +\infty} \lambda_n = 0$, there exists $n_0 \in \NN$ such that $2\eta - \lambda_n \geq 0$ for all $n\geq n_0$. Thus, for $\p{u, w} \in \Graph\p{A + D + N_C}$ such that $w = v + p + Du$, 
where $v\in Au$ and $p\in N_C\p{u}$, by Lemma \ref{lem:FBB:inequality} it holds for any $n\geq n_0$
\begin{align}\label{eq:FBB:modineq}
& \norm{x_{n} - u}^2 - \norm{x_{n-1} - u}^2 \leq \nonumber \\
& 2\lambda_n \beta_n \p{\sup_{\tilde u\in C} \varphi_B\p{\tilde u, \frac{p}{\beta_n}} - \sigma_C\p{\frac{p}{\beta_n}}} + 2\lambda_n \inpr{w}{u - x_{n}} + 2\lambda_n^2 \norm{Du + v}^2. 
  \end{align}

By Lemma \ref{lem:OpialPassty}, it is sufficient to prove that the following two statements hold:
  \begin{enumerate}
    \item \label{item:FBB:normconvergence} for every $u\in\zer\p{A+D+N_C}$ the sequence $\p{\norm{x_n - u}_{n\geq 0}}$ is convergent;
    \item \label{item:FBB:weakclusterpoints} every weak cluster point of $\p{z_n}_{n \in \NN}$ lies in $\zer\p{A+D+N_C}$.
  \end{enumerate}

\ref{item:FBB:normconvergence}. Let be an arbitrary $u\in \zer\p{A + D + N_C}$. By taking $w = 0$ in \eqref{eq:FBB:modineq}, we get
\begin{equation*}
\norm{x_{n} - u}^2 - \norm{x_{n-1} - u}^2 \leq 2\lambda_n \beta_n \p{\sup_{\tilde u\in C} \varphi_B\p{\tilde u, \frac{p}{\beta_n}} - \sigma_C\p{\frac{p}{\beta_n}}} + 2\lambda_n^2 \norm{Du + v}^2. 
  \end{equation*}
and the conclusion follows from Lemma \ref{lem:summability}.

\ref{item:FBB:weakclusterpoints}. Let $z$ be a weak cluster point of $\p{z_n}_{n \in \NN}$. As $A + D + N_C$ is maximally monotone, in order to show that $z\in\zer\p{A + D + N_C}$ 
we will use the characterization given in \eqref{eq:maxmon_zeros}. To this end we take $\p{u, w} \in \Graph\p{A + D + N_C}$ such that $w = v + p + Du$, where $v\in Au$ and $p\in N_C\p{u}$. Let $N\in \NN$ with $N \geq n_0 + 2$. Summing up for $n = n_0 + 1, \ldots, N$ the inequalities in \eqref{eq:FBB:modineq}, we get
 \begin{align*}
    \norm{x_{N} - u}^2 - \norm{x_{n_0} - u}^2 \leq L + 2\inpr{w}{\sum_{n = 1}^N \lambda_n u - \sum_{n = 1}^N \lambda_n x_{n}},
\end{align*}
 where
\begin{align*}
    L =  & 2\sum_{n=n_0 + 1}^\infty \lambda_n\beta_n \p{\sup_{\tilde u\in C} \varphi_B\p{\tilde u, \frac{p}{\beta_n}} - \sigma_C\p{\frac{p}{\beta_n}}} + 2\sum_{n = n_0 + 1}^\infty \lambda_n^2 \norm{Du + v}^2\\ 
         & + 2 \sum_{n=1}^{n_0} \lambda_n \inpr{w}{x_{n} - u}
\end{align*}
 is finite and independent from $N$. Discarding the nonnegative term $\norm{x_{N} - u}^2$ and dividing by $2\tau_N = 2 \sum_{n = 1}^N \lambda_n$, we obtain
  \[
    - \frac{\norm{x_{n_0} - u}^2}{2\tau_N} \leq \frac{L}{2\tau_N} + \inpr{w}{u - z_N}.
  \]
  By passing to the limit $N\to +\infty$ and using that $\lim_{N\to +\infty} \tau_N = +\infty$, we get
  \[
    \liminf_{N\to +\infty} \inpr{w}{u - z_N} \geq 0.
  \]
 Since $z$ is a weak cluster point of $\p{z_n}_{n \in \NN}$, we obtain that $\inpr{w}{u - z} \geq 0$. Finally, as this inequality holds for arbitrary $\p{u, w} \in \Graph\p{A + D + N_C}$, the desired conclusion follows.
\end{proof}

In the following we show that strong monotonicity of the operator $A$ ensures strong convergence of the sequence $(x_n)_{n \geq 0}$.

\begin{theorem}\label{str-mon-FBB} 
Let $(x_n)_{n \geq 0}$ and $(w_n)_{n \in \NN}$ be the sequences generated by Algorithm \ref{alg:FBB}. If \eqref{eq:Hfitz} is fulfilled and the operator $A$ is $\gamma$-strongly monotone with $\gamma>0$, 
then $(x_n)_{n \geq 0}$ converges strongly to the unique element in $\zer(A+D+N_C)$ as $n\rightarrow+\infty$.
\end{theorem}

\begin{proof} Let $u \in \zer(A+D+N_C)$ and $w=0=v+p+Du$, where $v \in Au$ and $p \in N_C(u)$. Since $A$ is $\gamma$-strongly monotone, inequality \eqref{Amon} becomes for any $n \in \NN$
\begin{equation}\label{Astrmon}
\inpr{y_{n-1} - w_n - \lambda_n v}{w_n - u} \geq \lambda_n\gamma\|w_n-u\|^2.
\end{equation}
Arguing as in the proof of Lemma \ref{lem:FBB:inequality} (for $w=0$) we obtain for any $n \in \NN$
\begin{align*}
 & \lambda_n\gamma\|w_n-u\|^2 + \norm{x_{n} - u}^2 - \norm{x_{n-1} - u}^2 + \lambda_n\p{2\eta - \lambda_n} \norm{Dx_{n-1} - Du}^2 + \\
 & \frac{1}{2}\norm{x_{n} - w_n}^2 + \frac{1}{2}\norm{x_{n} - w_n - \lambda_n \p{Du + v}}^2 + \norm{x_{n-1} - w_n - \lambda_n\p{Dx_{n-1} - Du}}^2  \leq\\
 & 2\lambda_n \beta_n \p{\sup_{\tilde u\in C} \varphi_B\p{\tilde u, \frac{p}{\beta_n}} - \sigma_C\p{\frac{p}{\beta_n}}} + 2\lambda_n^2 \norm{Du + v}^2. 
\end{align*}
As $\lim_{n\rightarrow+\infty}\lambda_n=0$, there exists $n_0\in\NN$ such that for all $n\geq n_0$
\begin{align*}
 & \lambda_n\gamma\|w_n-u\|^2 + \frac{1}{2}\norm{x_{n} - w_n}^2 + \norm{x_{n} - u}^2 - \norm{x_{n-1} - u}^2 \leq \\
 & 2\lambda_n \beta_n \p{\sup_{\tilde u\in C} \varphi_B\p{\tilde u, \frac{p}{\beta_n}} - \sigma_C\p{\frac{p}{\beta_n}}} + 2\lambda_n^2 \norm{Du + v}^2
\end{align*}
and, so,
\begin{align*}
 \gamma \sum_{n \geq n_0} \lambda_n\|w_n-u\|^2 + \frac{1}{2} \sum_{n \geq n_0} \norm{x_{n} - w_n}^2 & \leq \\
 \norm{x_{n_0} - u}^2 + 2 \sum_{n \geq n_0} \lambda_n \beta_n \p{\sup_{\tilde u\in C} \varphi_B\p{\tilde u, \frac{p}{\beta_n}} - \sigma_C\p{\frac{p}{\beta_n}}} + 2 \norm{Du + v}^2 \sum_{n \geq n_0} \lambda_n^2 & < +\infty. 
\end{align*}
Consequently, 
$\sum_{n \geq n_0} \lambda_n(\|x_n-u\| - \|x_n-w_n\|)^2 \leq \sum_{n \geq n_0} \lambda_n\|w_n-u\|^2 < +\infty$ and $\sum_{n \geq n_0} \norm{x_{n} - w_n}^2 < +\infty$. Since $(\|x_n-u\| - \|x_n-w_n\|)_{n \in \NN}$ is 
convergent (see the proof of Theorem \ref{thm:FBB}) and $\sum_{n \in \NN} \lambda_n=+\infty$, it follows that $\lim_{n\rightarrow+\infty} (\|x_n-u\| - \|x_n-w_n\|) = 0$ and, so,
$\lim_{n\rightarrow+\infty} \|x_n-u\|  = 0$.
\end{proof}

\section{A backward penalty scheme with two forward steps}
The problem we deal with in this section has the following formulation.
\begin{problem}\label{prob:FBFB}
  Let $\Hh$ be a real Hilbert space, $A, B: \Hh \toset \Hh$ be maximally monotone operators, $D: \Hh \to \Hh$ be an $\eta^{-1}$-Lipschitz continuous and monotone operator with $\eta > 0$ and suppose that $C 
  := \zer B \neq \emptyset$. The monotone inclusion problem to solve is
  \[
    0 \in Ax + Dx + N_C\p{x}.
  \]
\end{problem}

Problem \ref{prob:FBFB} is a generalization of Problem \ref{prob:FBB}, since every $\eta$-cocoercive operator is obviously monotone and $\eta^{-1}$-Lipschitz continuous. 
If $D = \nabla f$ for some convex and differentiable function $f: \Hh \to \RR$ with $\eta^{-1}$-Lipschitzian gradient, then 
$D$ is automatically $\eta$-cocoercive by the Baillon--Haddad theorem \cite{BaillonHaddad:1977}. The investigations we make in Section \ref{sec:PD} provide a strong motivation
for treating monotone inclusion problems in the setting of Problem \ref{prob:FBFB}.

\begin{algorithm}\label{alg:FBFB}
  Choose $x_1 \in \Hh$ and set for any $n \geq 1$
  \begin{align*}
    y_n &= x_n - \lambda_n Dx_n, \\
    p_n &= J_{\lambda_n A} y_n, \\
    q_n &= p_n - \lambda_n Dp_n, \\
    x_{n+1} &= J_{\lambda_n \beta_n B}\p{x_n - y_n + q_n},
  \end{align*}
  where $\p{\lambda_n}_{n \in \NN}$ and $\p{\beta_n}_{n \in \NN}$ are sequences of postive real numbers. For the convergence statement, the same additional hypotheses \eqref{eq:Hfitz} are needed as for Algorithm \ref{alg:FBB}.
\end{algorithm}

\begin{lemma} \label{lem:FBFB:inequality}
 For $u\in C\cap \dom A$ take $w \in \p{A + D + N_C}\p{u}$ such that $w = v + Du + p$ for some $v\in Au$ and $p\in N_C\p{u}$. 
 For each $n \in \NN$, the following inequality holds:
  \begin{align*}
    & \norm{x_{n+1} - u}^2 - \norm{x_n - u}^2 + \p{1 - \frac{4\lambda_n^2}{\eta^2}}\norm{x_n - p_n}^2 + \nonumber  \\
    & \frac{1}{2}\norm{x_{n+1} - p_n}^2 + \frac{1}{2} \norm{x_{n+1} - p_n + 2\lambda_n\p{Dp_n - Dx_n + p}}^2 \nonumber  \leq\\
    & 2\lambda_n \beta_n\p{\sup_{\tilde u\in C} \varphi_B\p{\tilde u, \frac{p}{\beta_n}} - \sigma_C\p{\frac{p}{\beta_n}}} + 2\lambda_n \inpr{w}{u - p_n} + 4 \lambda_n^2 \norm{p}^2.
  \end{align*}
\end{lemma}
\begin{proof} Let be $n \geq 1$ fixed. We have $\lambda_n v\in \lambda_n Au$ and $y_n - p_n \in \lambda_n A p_n$, so, by monotonicity of $\lambda_n A$,
  \[
    \inpr{y_n - p_n - \lambda_n v}{p_n - u} \geq 0,
  \]
  which is equivalent to
  \begin{equation}\label{eq:FBFB:monA}
    2\lambda_n \inpr{v}{p_n - u} \leq 2\inpr{y_n - p_n}{p_n - u} = \norm{y_n - u}^2 - \norm{y_n - p_n}^2 - \norm{p_n - u}^2.
  \end{equation}
Furthermore, we have $x_n - y_n + q_n - x_{n+1} \in \lambda_n \beta_n Bx_{n+1}$, so, by definition of the Fitzpatrick function,
  \begin{align}
    &\mathrel{\phantom{\leq}}2\lambda_n \beta_n \p{\sup_{\tilde u\in C} \varphi_B\p{\tilde u, \frac{p}{\beta_n}} - \sigma_C\p{\frac{p}{\beta_n}}} \nonumber \\
    &\geq 2\inpr{u}{x_n - y_n + q_n - x_{n+1}} + 2\lambda_n \inpr{p}{x_{n+1}} - 2\inpr{x_{n+1}}{x_n - y_n + q_n - x_{n+1}} \nonumber \\
    &\mathrel{\phantom{=}} \mathop{-} 2\lambda_n \inpr{p}{u} \nonumber \\
    &= 2\lambda_n \inpr{p}{x_{n+1} - u} + 2\inpr{u - x_{n+1}}{x_n - y_n + q_n - x_{n+1}} \nonumber \\
    &= 2\lambda_n \inpr{p}{x_{n+1} - u} - \norm{u - x_n}^2 + \norm{u - y_n}^2 - \norm{u - q_n}^2 + \norm{u - x_{n+1}}^2 \nonumber \\ 
    &\mathrel{\phantom{=}} + \norm{x_n - x_{n+1}}^2 - \norm{x_{n+1} - y_n}^2 + \norm{x_{n+1} - q_n}^2. \label{eq:FBFB:FitzB}
  \end{align}

 Adding \eqref{eq:FBFB:monA} and \eqref{eq:FBFB:FitzB}, we obtain
 {\allowdisplaybreaks \begin{align*}
    &\mathrel{\phantom{\leq}} \norm{x_{n+1} - u}^2 - \norm{x_n - u}^2 - 2\lambda_n \beta_n\p{\sup_{\tilde u\in C} \varphi_B\p{\tilde u, \frac{p}{\beta_n}} - \sigma_C\p{\frac{p}{\beta_n}}} - 2\lambda_n \inpr{w}{u - p_n} \nonumber \\
    &\leq 2 \lambda_n \inpr{p}{p_n - x_{n+1}} + 2\lambda_n \inpr{Du}{p_n - u} - \norm{y_n - p_n}^2 - \norm{u - p_n}^2 + \norm{u - q_n}^2 \nonumber \\ 
    &\mathrel{\phantom{=}} \mathop{-} \norm{x_{n+1} - x_n}^2 + \norm{x_{n+1} - y_n}^2 - \norm{x_{n+1} - q_n}^2 \nonumber \\
    &= 2 \lambda_n \inpr{p}{p_n - x_{n+1}} + 2\lambda_n \inpr{Du}{p_n - u} - \norm{x_n - \lambda_n Dx_n - p_n}^2 - \norm{u - p_n}^2 \nonumber \\
    &\mathrel{\phantom{=}} \mathop{+} \norm{u - p_n + \lambda_n Dp_n}^2 - \norm{x_{n+1} - x_n}^2 + \norm{x_{n+1} - x_n + \lambda_n Dx_n}^2 \nonumber\\ 
    &\mathrel{\phantom{=}} \mathop{-} \norm{x_{n+1} - p_n + \lambda_n Dp_n}^2 \nonumber \\
    &= 2 \lambda_n \inpr{p}{p_n - x_{n+1}} + 2\lambda_n \inpr{Du}{p_n - u} - \norm{x_n - p_n}^2 + 2\lambda_n\inpr{Dx_n}{x_n - p_n} \nonumber \\ 
    &\mathrel{\phantom{=}} \mathop{-}  \lambda_n^2 \norm{Dx_n}^2   + \lambda_n^2 \norm{Dp_n}^2 + 2\lambda_n\inpr{Dp_n}{u - p_n} + \lambda_n^2 \norm{Dx_n}^2 + 2\lambda_n \inpr{Dx_n}{x_{n+1} - x_n} \nonumber \\
    &\mathrel{\phantom{=}} \mathop{-} \norm{x_{n+1} - p_n}^2 - 2\lambda_n \inpr{x_{n+1} - p_n}{Dp_n} - \lambda_n^2 \norm{Dp_n}^2 \nonumber\\ 
    &= 2 \lambda_n \inpr{p}{p_n - x_{n+1}} + 2\lambda_n \inpr{Du}{p_n - u} - \norm{x_n - p_n}^2 + 2\lambda_n\inpr{Dp_n}{u - x_{n+1}} \nonumber \\
    &\mathrel{\phantom{=}} \mathop{+} 2\lambda_n \inpr{Dx_n}{x_{n+1} - p_n} - \norm{x_{n+1} - p_n}^2 \nonumber \\
    &= 2\lambda_n \inpr{Du - Dp_n}{p_n - u} + 2\lambda_n \inpr{Dp_n - Dx_n + p}{p_n - x_{n+1}} \nonumber\\ 
    &\mathrel{\phantom{=}} \mathop{-} \norm{x_{n+1} - p_n}^2 - \norm{x_n - p_n}^2 \nonumber \\
    &\leq - \norm{x_{n+1} - p_n}^2 - \norm{x_n - p_n}^2   + 2\lambda_n \inpr{Dp_n - Dx_n + p}{p_n - x_{n+1}} \nonumber\\ 
    &= -\norm{x_n - p_n}^2 - \frac{1}{2}\norm{x_{n+1} - p_n}^2 - \frac{1}{2} \norm{x_{n+1} - p_n + 2 \lambda_n\p{Dp_n - Dx_n + p}}^2 \nonumber \\
    &\mathrel{\phantom{=}} \mathop{+} 2 \lambda_n^2 \norm{Dp_n - Dx_n + p}^2 \nonumber \\
    &\leq -\p{1 - \frac{4 \lambda_n^2}{\eta^2}}\norm{x_n - p_n}^2 - \frac{1}{2}\norm{x_{n+1} - p_n}^2 \nonumber \\
    &\mathrel{\phantom{=}} \mathop{-} \frac{1}{2} \norm{x_{n+1} - p_n + 2 \lambda_n \p{Dp_n - Dx_n + p}}^2 + 4 \lambda_n^2 \norm{p}^2,
  \end{align*}}
which leads to the desired conclusion.  
\end{proof}

\begin{theorem}\label{thm:FBFB}
 Let $\p{x_n}_{n \in \NN}$, $\p{p_n}_{n \in \NN}$, $\p{y_n}_{n \in \NN}$ and $\p{q_n}_{n \in \NN}$ be the sequences generated by Algorithm \ref{alg:FBFB} and $\p{z_n}_{n \in \NN}$ be the sequence defined in \eqref{eq:zndef}. 
 If \eqref{eq:Hfitz} is fulfilled, then $\p{z_n}_{n \in \NN}$ converges weakly to an element in $\zer\p{A + D + N_C}$ as $n\to\infty$.
\end{theorem}

\begin{proof}
As $\lim_{n\to +\infty} \lambda_n = 0$, there exists $n_0 \in \NN$ such that $1 - \frac{4 \lambda_n^2}{\eta^2} \geq 0$ for all $n\geq n_0$. Thus, for $\p{u, w} \in \Graph\p{A + D + N_C}$ such that $w = v + p + Du$, 
where $v\in Au$ and $p\in N_C\p{u}$, by Lemma \ref{lem:FBFB:inequality} it holds for any $n\geq n_0$
\begin{align}\label{eq:FBFB:modineq}
& \norm{x_{n+1} - u}^2 - \norm{x_n - u}^2  \leq \nonumber  \\
& 2\lambda_n \beta_n\p{\sup_{\tilde u\in C} \varphi_B\p{\tilde u, \frac{p}{\beta_n}} - \sigma_C\p{\frac{p}{\beta_n}}} + 2\lambda_n \inpr{w}{u - p_n} + 4 \lambda_n^2 \norm{p}^2.
\end{align}

Analogously to the proof of Theorem \ref{thm:FBB}, one obtains from here that:
  \begin{enumerate}
    \item \label{item:FBFB:normconvergence} for every $u\in\zer\p{A+D+N_C}$ the sequence $\p{\norm{x_n - u}_{n \in \NN}}$ is convergent;
    \item \label{item:FBFB:weakclusterpoints} every weak cluster point of $\p{z_n}_{n \in \NN}$ lies in $\zer\p{A+D+N_C}$.
  \end{enumerate}
The conclusion follows by using again Lemma \ref{lem:OpialPassty}.
\end{proof}

Arguing in the same way as in Theorem \ref{str-mon-FBB}, one can show that the strong monotonicity of $A$ guarantees strong convergence of the sequence $(x_n)_{n \in \NN}$.

\begin{theorem}\label{str-mon-FBFB} 
Let $\p{x_n}_{n \in \NN}$, $\p{p_n}_{n \in \NN}$, $\p{y_n}_{n \in \NN}$ and $\p{q_n}_{n \in \NN}$ be the sequences generated by Algorithm \ref{alg:FBFB}. 
If \eqref{eq:Hfitz} is fulfilled and the operator $A$ is $\gamma$-strongly monotone with $\gamma>0$, then $(x_n)_{n \in \NN}$ converges strongly to the unique element in $\zer(A+D+N_C)$ as $n\rightarrow+\infty$.
\end{theorem}

\section{A primal-dual algorithm based on a backward penalty scheme}\label{sec:PD}
In this section, we will derive a primal-dual algorithm for solving complexly structured monotone inclusion problems based on the backward penalty iterative scheme provided in Algorithm \ref{alg:FBFB}. The
problem under investigation is the following one.

\begin{problem}\label{prob:PD}
 Let $\Hh$ be a real Hilbert space, $A: \Hh \toset \Hh$ a maximally monotone operator and $D: \Hh \to \Hh$ a monotone and $\eta^{-1}$-Lipschitz continuous operator for some $\eta > 0$. Furthermore, 
 let $m\geq 1$ and for any $i \in \set{1, \ldots, m}$, let $\Gg_i$ be real Hilbert spaces, $A_i: \Gg_i \toset \Gg_i$ maximally monotone operators, $D_i: \Gg_i \toset \Gg_i$ be maximally monotone operators such that 
$D_i^{-1}$ are $\nu_i^{-1}$-Lipschitz continuous for some $\nu_i > 0$ and $L_i: \Hh\to \Gg_i$  nonzero linear continuous operators.  Consider also $B: \Hh \toset \Hh$ a maximally monotone operator and 
suppose that $C := \zer B \neq \emptyset$. The monotone inclusion problem to solve is to find $x\in \Hh$ with
  \begin{equation}\label{eq:PD:primal}
    0 \in Ax + \sum_{i = 1}^m L_i^* \p{A_i \bop D_i}\p{L_i x} + Dx + N_C\p{x},
  \end{equation}
together with its dual monotone inclusion problem in the sense of Attouch--Th\'era \cite{AttouchThera:1996} of finding $v_i \in \Gg_i$, $i = 1, \ldots, m$, satisfying
  \begin{equation}\label{eq:PD:dual}
    \exists x \in \Hh: \qquad v_i\in \p{A_i \bop D_i}\p{L_i x} \quad \text{and} \quad 0 \in Ax + \sum_{i = 1}^m L_i^* v_i + Dx + N_C\p{x}.
  \end{equation}
\end{problem}

We introduce the real Hilbert space $\bm\Hh \defeq \Hh \times \Gg_1 \times \ldots \times \Gg_m$, the operators
\begin{align}\label{defoper}
\bm A : \bm \Hh \toset \bm \Hh, \  \bm A\p{x, v_1, \ldots, v_m} &= Ax \times A_1^{-1} v_1 \times \ldots \times A_m^{-1} v_m, \nonumber \\
\bm D: \bm \Hh \to \bm \Hh, \ \bm D\p{x, v_1, \ldots, v_m} &= \p{\sum_{i=1}^m L_i^* v_i + Dx, D_1^{-1} v_1 - L_1 x, \ldots, D_m^{-1} v_m - L_m x}, \nonumber \\
\bm B: \bm \Hh \toset \bm \Hh, \  \bm B\p{x, v_1, \ldots, v_m} &= Bx \times \set{0} \times \ldots \times \set{0},
\end{align}
and the set 
$$\bm C :=\setcond{\bm x \in \bm \Hh}{\bm B \bm x = \bm 0} = \zer B \times \Gg_1 \times \ldots \times \Gg_m.$$ 
In this setting, we have for $\bm x = \p{x, v_1, \ldots, v_m} \in \Hh$
\begin{align}
  \bm 0 \in \p{\bm A + \bm D +  N_{\bm C}}\bm x &\iff \set{\begin{matrix}0 \in Ax + \sum_{i=1}^m L_i^* v_i + Dx + N_C\p{x} \\ 0\in A_1^{-1} v_1 + D_1^{-1} v_1 - L_1 x \\ \vdots \\ 0\in A_m^{-1} v_m + D_m^{-1} v_m - L_m x\end{matrix}} \nonumber \\
  &\iff \set{\begin{matrix}0 \in Ax + \sum_{i=1}^m L_i^* v_i + Dx  + N_C\p{x} \\ v_i \in \p{A_i \bop D_i}\p{L_i x}, i = 1, \ldots, m\end{matrix}} \nonumber \\
  &\ \Longrightarrow x \text{ satisfies } \eqref{eq:PD:primal} \text{ and } \p{v_1, \ldots, v_m} \text{ satisfies } \eqref{eq:PD:dual}. \label{eq:PD:equivalence}
\end{align}
The resolvent of $\bm B$ is given by
\[
  J_{\gamma \bm B}\p{x, v_1, \ldots, v_m} = \p{J_{\gamma B}x, v_1, \ldots, v_m}
\]
and its Fitzpatrick function $\varphi_{\bm B} : \bm \Hh \times \bm \Hh \rightarrow \overline \RR$  by
\begin{align*}
  \varphi_{\bm B}\p{x, v_1, \ldots, v_m, x^*, v_1^*, \ldots, v_m^*} &= \sup_{\substack{y\in\Hh \\ y^* \in By \\ w_i \in \Gg_i, i = 1,\ldots, m}}\hspace{-0.5cm} \set{\inpr{x}{y^*} + \inpr{y}{x^*} + \sum_{i=1}^m \inpr{w_i}{v_i^*} - \inpr{y}{y^*}} \\
  &= \begin{cases}\varphi_B\p{x, x^*},& \text{if } v_i^* = 0, \ i = 1, \ldots, m, \\ +\infty,&\text{otherwise.}\end{cases}
\end{align*}
Thus, in order to solve the primal-dual pair of monotone inclusion problems \eqref{eq:PD:primal}--\eqref{eq:PD:dual}, one has to solve
\begin{equation}\label{probprodspace}
\bm 0 \in \bm A \bm x + \bm D \bm x +  N_{\bm C}(\bm x)
\end{equation}
 in the product space $\bm \Hh$. By doing this via Algorithm \ref{alg:FBFB} one obtains the following iterative scheme:
\begin{algorithm}\label{alg:PD}
  Choose $x_1 \in \Hh$ and $v_{i, 1} \in \Gg_i$, $i = 1, \ldots, m$, and set for any $n \geq 1$
  \begin{align*}
    y_{1, n} &= x_n - \lambda_n\p{\sum_{i=1}^m L_i^* v_{i, n} - Dx_n}, \\
    y_{2, i, n} &= v_{i, n} - \lambda_n \p{D_i^{-1} v_{i, n} - L_i x_n}, i = 1, \ldots, m, \\
    p_{1, n} &= J_{\lambda_n A} y_{1, n}, \\
    p_{2, i, n} &= J_{\lambda_n A_i^{-1}} y_{2, i, n}, i = 1, \ldots, m, i = 1, \ldots, m, \\
    q_{n} &= p_{1, n} - \lambda_n \p{\sum_{i=1}^m L_i^* p_{2, i, n} - Dp_{1, n}}, \\
    v_{i, n+1} &= v_{i, n} - y_{2, i, n} + p_{2, i, n} - \lambda_n \p{D_i^{-1} p_{2, i, n} - L_i p_{1, n}}, i = 1, \ldots, m, \\
    x_{n+1} &= J_{\lambda_n \beta_n B}\p{x_n - y_{1, n} + q_n}.
  \end{align*}
where $\p{\lambda_n}_{n \in \NN}$ and $\p{\beta_n}_{n \in \NN}$ are sequences of postive real numbers.
\end{algorithm}
For its convergence the following hypotheses are needed:
\begin{equation}\label{eq:Hfitz:comp}
  \left.\begin{minipage}{14.2cm}
    \begin{enumerate}
      \item[(i)]  $A + N_C$ is maximally monotone and $\zer\p{A + \sum\limits_{i = 1}^m L_i^* \circ \p{A_i \bop D_i} \circ L_i + D + N_C}$ $\neq \emptyset$;
      \item [(ii)] For every $p\in\Ran N_C$, $\displaystyle\sum_{n\in \NN} \lambda_n \beta_n \p{\sup_{\tilde u\in C} \varphi_B\p{\tilde u, \frac{p}{\beta_n}} - \sigma_C\p{\frac{p}{\beta_n}}} < +\infty$;
      \item [(iii)] $\p{\lambda_n}_{n\geq 0} \in \ell^2 \setminus \ell^1$.
    \end{enumerate}
  \end{minipage}
  \right\rbrace \tag{H\textsubscript{pd}}
\end{equation}

\begin{theorem}\label{thm:PD}
Consider the sequences generated by Algorithm \ref{alg:PD} and assume that \eqref{eq:Hfitz:comp} is fulfilled. Then the sequence $\p{z_n}_{n \in \NN}$ defined in \eqref{eq:zndef} converges weakly to a solution of \eqref{eq:PD:primal} 
and $\p{\frac{1}{\sum_{k = 1}^n \lambda_k} \sum_{k = 1}^n \lambda_k (v_{1, k},..,v_{m,k})}_{n \in \NN}$ converges weakly to a solution of \eqref{eq:PD:dual} as $n \rightarrow +\infty$. If, additionally, $A$ and $A_i^{-1}, i=1,...,m,$ 
are strongly monotone, then $(x_n)_{n \in \NN}$ converges strongly to the unique solution of \eqref{eq:PD:primal} and $(v_{1, n},..,v_{m, n})_{n \in \NN}$ converges strongly to the unique 
solution of \eqref{eq:PD:dual} as $n \rightarrow +\infty$.
\end{theorem}
\begin{proof}
  Clearly, the iterations in Algorithm \ref{alg:PD} can be for any $n \geq 1$ equivalently written as 
  \begin{align*}
    \p{y_{1, n}, y_{2, 1, n}, \ldots, y_{2, m, n}} &= \p{\bm{\Id} - \lambda_n \bm D}\p{x_n, v_{1, n}, \ldots, v_{m, n}}, \\
    \p{p_{1, n}, p_{2, 1, n}, \ldots, p_{2, m, n}} &= J_{\bm{A}}\p{y_{1, n}, y_{2, 1, n}, \ldots, y_{2, m, n}}, \\
    \p{q_n, \tilde q_{1, n}, \ldots, \tilde q_{m, n}} &= \p{\bm{\Id} - \lambda_n \bm D}\p{p_{1, n}, p_{2, 1, n}, \ldots, p_{2, m, n}}, \\
    \p{x_{n+1}, v_{1, n+1}, \ldots, v_{m, n+1}} &= J_{\lambda_n \beta_n \bm B}\Bigl(\p{x_n, v_{1, n}, \ldots, v_{m, n}} - \p{y_{1, n}, y_{2, 1, n}, \ldots, y_{2, m, n}} \\ &\mathrel{\phantom{=}} \mathop{+} \p{q_n, \tilde q_{1, n}, \ldots, \tilde q_{m, n}}\Bigr).
  \end{align*}
 with the operators $\bm{A}$, $\bm{B}$ and $\bm{D}$ defined in \eqref{defoper}. The operators $\bm{A}$ and $\bm{B}$ are maximally monotone by \cite[Proposition 20.23]{BauschkeCombettes:2011}, and the operator $\bm{D}$ 
 is monotone and Lipschitz continuous (\cite{CombettesPesquet:2012}). If   $A$ and $A_i^{-1}, i=1, \ldots, m,$  are strongly monotone, then $\bm{A}$ is strongly monotone, too. Thus the conclusion is a direct consequence of the Theorem \ref{thm:FBFB} and Theorem \ref{str-mon-FBFB} applied to the monotone inclusion problem \eqref{probprodspace}, provided that the corresponding hypotheses  \eqref{eq:Hfitz} are fulfilled.

According to \eqref{eq:Hfitz:comp}, $A + N_C$ is maximally monotone, and so is $\bm A + N_{\bm C}$. Further, from $\zer\p{A + \sum\limits_{i = 1}^m L_i^* \circ \p{A_i \bop D_i} \circ L_i + D + N_C}$ $\neq \emptyset$ it follows that $\zer\p{\bm A + \bm D+  N_{\bm C}}$ $\neq \emptyset$.

Furthermore, $\Ran N_{\bm C} = \Ran N_C \times \set{0} \times \ldots \times \set{0}$, so for all $\p{p, 0, \ldots, 0} \in \Ran N_{\bm C}$
  \begin{align*}
    \displaystyle\sum_{n\in \NN} \lambda_n \beta_n \sup_{\p{\tilde u, v_1, \ldots, v_m} \in \bm C}\p{\varphi_{\bm B}\p{\tilde u, v_1, \ldots, v_m, \frac{p}{\beta_n}, 0, \ldots, 0} - \sigma_{\bm C}\p{\frac{p}{\beta_n}, 0, \ldots, 0}}  &  = \\
    \displaystyle\sum_{n\in \NN} \lambda_n \beta_n \sup_{\tilde u \in C} \p{\varphi_B\p{\tilde u, \frac{p}{\beta_n}} - \sigma_C\p{\frac{p}{\beta_n}}} & < +\infty. 
    \end{align*}
\end{proof}

\begin{remark}\label{fbbvsfbfb}
Even if the operators $D$ and $D_i^{-1}, i=1,...,m,$ are cocoercive, one cannot make use of Algorithm \ref{alg:FBB} and of the corresponding convergence theorem in order to solve the monotone inclusion problem \eqref{probprodspace}. This is due to the fact that the operator
$$(x,v_1,...,v_m) \mapsto   \p{\sum_{i=1}^m L_i^* v_i, - L_1 x, \ldots, - L_m x}, $$
being skew, fails to be cocoercive, which means that $D$ is not cocoercive as well. This shows the importance of having iterative schemes for monotone inclusion problems involving monotone and Lipschitz continuous  operators, which are not necessarily cocoercive, as is Algorithm \ref{alg:FBFB} (see, also \cite{BotCsetnek:2013b}).
\end{remark}

\section{Examples}\label{sec:Examples}
In this section,  we discuss the fulfillment of condition (ii) in the hypotheses \eqref{eq:Hconj} and \eqref{eq:Hfitz}, for several particular instances of the operator $B$.

\begin{example}\label{ex1}
  For a convex and closed set $\emptyset \neq C \subseteq \Hh$, let $B := N_C$. Then $\zer B=C$ and (see \cite[Example 3.1]{BauschkeMclarenSendov:2006})
  \[
    \varphi_B\p{x, u} = \varphi_{N_C}\p{x, u}  = \delta_C\p{x} + \sigma_C\p{u},
  \]
  and condition (ii) in \eqref{eq:Hfitz} becomes
  \[
    \sum_{n\in \NN} \lambda_n\beta_n \p{\sup_{\tilde u \in C} \delta_C\p{\tilde u} + \sigma_C\p{\frac{p}{\beta_n}} - \sigma_C\p{\frac{p}{\beta_n}}} < +\infty,
  \]
 which is satisfied for any choice of the sequences $\p{\lambda_n}_{n \in \NN}$ and $\p{\beta_n}_{n \in \NN}$. The same applies for condition (ii) in \eqref{eq:Hconj}, where $\Psi\p{x} = \delta_C\p{x}$ and $\Psi^*\p{u} = \sigma_C\p{u}$.
\end{example}

\begin{example}\label{ex:distsq}
For a convex and closed set $\emptyset \neq C \subseteq \Hh$,  let $\Psi :  \Hh \rightarrow \RR$, $\Psi\p{x} = \frac{1}{2} d_C\p{x}^2$, where $d_C(x) = \inf_{z \in C} \|x-z\|$ and 
$B := \partial \Psi$. Then $\zer B = C$ and (see \cite[Corollary 12.30]{BauschkeCombettes:2011})
  \[
    \nabla d_C\p{x} = x - \Proj_C\p{x},
  \]
where $\Proj_C :  \Hh \rightarrow C$ denotes the projection operator on $C$. We have $\Psi = \delta_C \bop \p{\frac{1}{2} \norm{\cdot}^2}$, so $\Psi^* = \sigma_C + \frac{1}{2} \norm{\cdot}^2$. 
If $C\neq \Hh$, condition (ii) in \eqref{eq:Hconj} is therefore equivalent to (see \cite{AttouchCzarneckiPeypouquet:2011a})
  \begin{equation}\label{eq:relpro}
    \sum_{n \in \NN} \frac{\lambda_n}{\beta_n} < +\infty,
  \end{equation}
in which case  condition (ii) in \eqref{eq:Hfitz} is also fulfilled. Let also notice that the resolvent of $B$ is given by
\[
    J_{\gamma B}\p{x} = \frac{x}{\gamma + 1} + \frac{\gamma \Proj_C\p{x}}{\gamma + 1} \ \forall x \in \Hh.
  \]
\end{example}
Next, we present two examples, for which condition (ii) in \eqref{eq:Hfitz} fails for any choice of the sequence of positive penalty parameters $\p{\beta_n}_{n \in \NN}$.
\begin{example}
For a convex and closed set $\emptyset \neq C \subseteq \Hh$, let $\Psi :  \Hh \rightarrow \RR$, $\Psi\p{x} = d_C\p{x}$, and $B := \partial \Psi$. Then
(see \cite[Example 16.49]{BauschkeCombettes:2011})  
  \[
    Bx = \partial d_C \p{x} = \begin{cases}\setcond{u\in N_C\p{x}}{\norm{u} \leq 1},& \text{if } x\in C, \\ \set{\frac{x - \Proj_C\p{x}}{\norm{x - \Proj_C\p{x}}}}, & \text{otherwise.}\end{cases}
  \]
and $\zer B = C$. Since $\Psi^* = \sigma_C + \delta_{\BB}$,
  \begin{equation}\label{eq:distfunc}
    \sum_{n \in \NN} \lambda_n \beta_n \p{\Psi^*\p{\frac{p}{\beta_n}} - \sigma_C\p{\frac{p}{\beta_n}}} = \sum_{n \in \NN} \lambda_n \beta_n \delta_{\BB}\p{\frac{p}{\beta_n}}.
  \end{equation}
For $C \neq \Hh$ and arbitrary $\beta_n > 0$, with $n \in \NN$, there exists $p\in \Ran N_C$ with $\norm{p} > \beta_n$, for which expression \eqref{eq:distfunc} is equal to $+\infty$. Thus, condition (ii) in \eqref{eq:Hconj} 
is not verified.

Condition (ii) in \eqref{eq:Hfitz} fails for similar reasons. Let be $\beta_n > 0$, with $n \in \NN$, $y\in C$ and $p\in N_C\p{y}$ with $\norm{p} > \beta_n$. Then
  \[
    \p{y + tp, \frac{p}{\norm{p}}} \in \Graph \partial d_C \ \forall t > 0,
  \]
  which implies that
  \begin{align*}
    &\mathrel{\phantom{=}} \sup_{\tilde u \in C} \varphi_{\partial d_C}\p{\tilde u, \frac{p}{\beta_n}} - \sigma_C\p{\frac{p}{\beta_n}} \geq \varphi_{\partial d_C}\p{y, \frac{p}{\beta_n}} - \sigma_C\p{\frac{p}{\beta_n}} \\
    & \geq  \sup_{t >0} \p{\inpr{y}{\frac{p}{\norm{p}}} + \inpr{y + tp}{\frac{p}{\beta_n}} - \inpr{y + tp}{\frac{p}{\norm{p}}} - \inpr{y}{\frac{p}{\beta_n}}} \\
    &= \sup_{t >0} \p{t\norm{p}\p{\frac{\norm{p}}{\beta_n} - 1}} = +\infty.
  \end{align*}
\end{example}

\begin{remark}\label{comparison}
One can notice that in the previous three examples, despite of the different choices of the operator $B$, the set of its zeros is the convex and closed set $C$. In what concerns condition (ii) in \eqref{eq:Hfitz}
and \eqref{eq:Hconj}, it is satisfied for any choice of $\p{\lambda_n}_{n \in \NN}$ and $\p{\beta_n}_{n \in \NN}$ when $B=N_C$ and for the two sequences fulfilling assumption \eqref{eq:relpro} when $B=\partial 
\left(\frac{1}{2} d_C^2 \right)$, however, it fails for any choice of $\p{\lambda_n}_{n \in \NN}$ and $\p{\beta_n}_{n \in \NN}$ when $B=\partial d_C$. 
The applicability of \eqref{eq:Hfitz} and \eqref{eq:Hconj} therefore depends on the modelling of the variational inequality via the set-valued operator $B$.
\end{remark}

\begin{example}
  Let $B: \Hh \to \Hh$ be a nonzero skew linear continuous operator. So $\zer B = \ker B$ and by taking $p\in \Ran N_{\ker B} = \p{\ker B}^\perp \setminus \{0\}$, it holds for any $n \in \NN$
  \begin{align*}
    &\mathrel{\phantom{=}} \sup_{\tilde u \in \ker B} \varphi_B\p{\tilde u, \frac{p}{\beta_n}} - \sigma_{\ker B}\p{\frac{p}{\beta_n}} \\
    &= \sup_{\tilde u\in \ker B} \sup_{y\in \Hh} \p{\inpr{\tilde u}{By} + \inpr{y}{\frac{p}{\beta_n}}- \inpr{y}{By}} \\
    & \geq  \sup_{y\in \Hh} \inpr{y}{\frac{p}{\beta_n}} = +\infty.
  \end{align*}
This shows that condition (ii) in \eqref{eq:Hfitz} is not satisfied.
\end{example}

\end{document}